%% file: paper.tex
\documentclass
[
    fontsize = 11 pt,
    american,
    captions = tableheading,
    numbers = noenddot,
    abstracton,
    paper = letter,
    DIV=13,
    left=1in,
    right=1in,
    top=1in,
    bottom=1in,
]
{scrartcl}

\input{core/standard_packages_2.tex}

\input{packages_and_commands/additional_packages.tex}
\input{core/general_commands_2.tex}
\input{packages_and_commands/additional_commands.tex}

\usepackage[utf8]{inputenc}

\title{Reemergence of the Epidemic Threshold in SIRS Infections on Connected Stars}
\author{%
    Andreas Göbel$^{*}$ \and Nicolas Klodt$^{*}$ \and Martin~S. Krejca$^{\dagger}$
}

\publishers
 {
  	\footnotesize $^{*}$ Hasso Plattner Institute, University of Potsdam, Potsdam, Germany\\
  	\{andreas.goebel, nicolas.klodt\}@hpi.de \\[1 ex]

  	$^{\dagger}$ LIX, CNRS, Ecole Polytechnique, Institut Polytechnique de Paris, Palaiseau, France \\
  	martin.krejca@polytechnique.edu
 }

\usepackage{graphicx}
\usepackage{mathtools}
\usepackage{amsmath}
\usepackage{amssymb}
\usepackage{enumitem}
\usepackage{amsthm}
\usepackage{hyperref}
\usepackage{svg}

\begin{document}

\maketitle

\begin{abstract}
    The SIRS process is a continuous-time process for how infections spread on a graph.
    In this model, each vertex is in one of the following three states: susceptible (to the infection; S), infected~(I), or recovered~(R) and thus immune to the infection.
    For each vertex, the transition among these states is exponentially distributed according to the parameters of the process.
    It was recently shown that recovered vertices effectively stop the infection on stars, that is, the expected survival time of SIRS processes on stars is bounded from above by a polynomial in the number of the vertices, independently of the infection rate of the process. The setting where the process has, so far, been shown to exhibit epidemic behavior, i.e., super-polynomial survival time when the infection rate is above some threshold value, requires the host graph to be an expander. This is in contrast to the shown behavior of the well-studied SIS process, a related model in which vertices never transition to~R, and in which even sparsely connected graphs, in particular stars, exhibit epidemic behavior.

    In this work, we show that expansion of the host graph is not a necessary condition for the SIRS process to result in an epidemic. Our main technical contribution shows that, while the SIRS process does not survive super-polynomially long on a single star, it does so on a network of poly-logarithmic (in the total number of vertices) stars of polynomial size. In addition, we show that such substructures appear in popular complex network models, providing for each a bound on the epidemic threshold. In particular, on hyperbolic random graphs, we compare our threshold for connected stars with the previously known one based on expansion, finding that both of them can be more permissive depending on the graph's power-law exponent and the rate that determines how long immunity lasts.
\end{abstract}

\section{Introduction}
\input{content/introduction}

\section{The SIRS Process and Complex Graph Network Models}
\label{sec:preliminaries}
\input{content/prelims}

\section{Lower Bound of the SIRS Survival Time on Connected Stars}
\label{sec:theory}
\input{content/SIRS_connected_stars}

\section{Super-Polynomial Survival Time of the SIRS Process in Scale-Free Graph Models}\label{sec:real_world_graphs}
\input{content/application}

\section{Outlook}
\label{sec:outlook}
\input{content/outlook}

\section*{Acknowledgments}
\input{content/acknowledgments}

\printbibliography

\newpage

\appendix

\input{content/appendix}

\end{document}

%% file: core/standard_packages_2.tex
\usepackage[T1]{fontenc}
\usepackage[utf8]{inputenc}
\usepackage
[
    babel = true, 
]
{microtype}           
\usepackage{csquotes} 

\usepackage[dvipsnames]{xcolor} 

\definecolor{stroke1}{HTML}{2574A9} 

\usepackage{amsmath}
\usepackage{amssymb}
\usepackage{amsthm}
\usepackage{thmtools}
\usepackage{mathtools}
\usepackage{thm-restate}
\usepackage{dsfont}        
\usepackage{braceMnSymbol} 

\usepackage
[
    ttscale = 0.85, 
]
{libertine} 
\usepackage
[
    libertine,    
    slantedGreek, 
    vvarbb,       
    libaltvw,     
]
{newtxmath} 
\usepackage{url} 
\usepackage{bm}  

\usepackage{graphicx} 
\usepackage
[
    subrefformat = simple, 
    labelformat = simple,  
]
{subcaption}         

\usepackage{array}     
\usepackage{booktabs}  
\usepackage{longtable} 
\usepackage{pdflscape} 

\usepackage[shortlabels]{enumitem} 

\usepackage
[
    ruled,         
    vlined,        
    linesnumbered, 
]
{algorithm2e} 

\usepackage{xspace}   
\usepackage
[
    shortcuts, 
]
{extdash}             
\usepackage{setspace} 

\usepackage{xparse}    
\usepackage{footnote}  
\usepackage{afterpage} 
\usepackage
[
    textsize = scriptsize, 
]
{todonotes}            

\usepackage
[
    sortcites,              
    style = alphabetic,     
    defernumbers,           
    safeinputenc,           
    backref = true,         
    backrefstyle = three,   
    hyperref = true,        
    maxbibnames = 99,       
    mincitenames = 1,       
    maxcitenames = 3,       
    minalphanames = 3,       
]
{biblatex} 

\DefineBibliographyStrings{english}%
{%
    backrefpage  = {\lowercase{s}ee page}, 
    backrefpages = {\lowercase{s}ee pages} 
}

\input{packages_and_commands/additional_packages}

\usepackage
[
    bookmarks = true,                 
    bookmarksopen = false,            
    bookmarksnumbered = true,         
    pdfstartpage = 1,                 
    colorlinks = true,                
    allcolors = stroke1,              
]
{hyperref} 

\usepackage
[
    noabbrev,   
    nameinlink, 
]
{cleveref} 

%% file: packages_and_commands/additional_packages.tex
\usepackage{multirow}
\usepackage{caption}

%% file: core/general_commands_2.tex
\date{}

\makeatletter
    \def\IfEmptyTF#1%
    {%
        \if\relax\detokenize{#1}\relax%
            \expandafter\@firstoftwo%
        \else%
            \expandafter\@secondoftwo%
        \fi%
    }
\makeatother

\NewDocumentCommand{\mathOrText}{m}
{%
    \ensuremath{#1}\xspace%
}

\let\originalleft\left
\let\originalright\right
\renewcommand{\left}{\mathopen{}\mathclose\bgroup\originalleft}
\renewcommand{\right}{\aftergroup\egroup\originalright}

\makeatletter
    \DeclareRobustCommand{\bfseries}%
    {%
        \not@math@alphabet\bfseries\mathbf%
        \fontseries\bfdefault\selectfont%
        \boldmath%
    }

\xspaceaddexceptions{]\}}

\allowdisplaybreaks

\crefname{ineq}{inequality}{inequalities}
\creflabelformat{ineq}{#2{\upshape(#1)}#3}

\crefname{term}{term}{terms}
\creflabelformat{term}{#2{\upshape(#1)}#3}

\crefname{cond}{condition}{conditions}
\creflabelformat{cond}{#2{\upshape(#1)}#3}

\crefname{assume}{assumption}{assumptions}
\creflabelformat{assume}{#2{\upshape(#1)}#3}

\let\oldfootnote\footnote

\newlength{\spaceBeforeFootnote} 
\newlength{\spaceAfterFootnote}  

\RenewDocumentCommand{\footnote}{o o o m}%
{%
    \IfNoValueTF{#1}%
    {%
        \oldfootnote{#4}%
    }%
    {%
        \setlength{\spaceBeforeFootnote}{\IfEmptyTF{#1}{0}{#1} em}%
        \IfNoValueTF{#2}%
        {%
            \hspace*{\spaceBeforeFootnote}\oldfootnote{#4}%
        }%
        {%
            \setlength{\spaceAfterFootnote}{\IfEmptyTF{#2}{0}{#2} em}%
            \hspace*{\spaceBeforeFootnote}\IfNoValueTF{#3}{\oldfootnote{#4}}{\oldfootnote[#3]{#4}}\hspace*{\spaceAfterFootnote}%
        }%
    }%
}

\makesavenoteenv{figure}
\makesavenoteenv{table}
\makesavenoteenv{tabular}

\declaretheoremstyle
[
   	spaceabove = \topsep,
   	spacebelow = \topsep,
   	headfont = \bfseries,
   	headformat = \textcolor{stroke1}{$\blacktriangleright$} \NAME~\NUMBER \NOTE,
   	notefont = \bfseries,
   	notebraces = {(}{)},
   	bodyfont = \normalfont,
   	postheadspace = 0.5 em,
   	qed = \textcolor{stroke1}{\bfseries$\blacktriangleleft$},
]
{myTheoremStyle}

\declaretheorem
[
   	style = myTheoremStyle,
   	name = Lemma,
    sharenumber = conjecture,
]
{lemma}
\declaretheorem
[
   	style = myTheoremStyle,
   	name = Corollary,
    sharenumber = conjecture,
]
{corollary}
\declaretheorem
[
   	style = myTheoremStyle,
   	name = Theorem,
    sharenumber = conjecture,
]
{theorem}
\declaretheorem
[
   	style = myTheoremStyle,
   	name = Definition,
    sharenumber = conjecture,
]
{definition}

\NewDocumentCommand{\functionTemplate}{m m m m o}%
{%
    \IfNoValueTF{#5}%
    {%
        \mathOrText{#1\left#2{#4}\right#3}%
    }%
    {%
        \mathOrText{#1#5#2{#4}#5#3}%
    }%
}

\newcommand*{\leftBracketType}{(}
\newcommand*{\rightBracketType}{)}

\NewDocumentCommand{\createFunction}{m m o o}%
{%
    \renewcommand*{\leftBracketType}{\IfNoValueTF{#3}{(}{#3}}%
    \renewcommand*{\rightBracketType}{\IfNoValueTF{#4}{)}{#4}}%
    \NewDocumentCommand{#1}{o o}%
    {%
        \IfNoValueTF{##1}%
        {%
            \mathOrText{#2}%
        }%
        {%
            \functionTemplate{#2}{\leftBracketType}{\rightBracketType}{##1}[##2]%
        }%
    }%
}

\DeclareDocumentCommand{\probabilisticFunctionTemplate}{m m O{} o}
{%
    \functionTemplate{#1}%
    {\lbrack}%
    {\rbrack}%
    {#2\IfEmptyTF{#3}{}{\ \IfNoValueTF{#4}{\left}{#4}\vert\ \vphantom{#2}#3\IfNoValueTF{#4}{\right.}{}}}%
    [#4]%
}

\newcommand*{\N}{\mathOrText{\mathds{N}}}

\newcommand*{\R}{\mathOrText{\mathds{R}}}

\newcommand*{\indicatorFunctionSymbol}{\mathds{1}}

\RenewDocumentCommand{\Pr}{m O{} o}%
{%
    \probabilisticFunctionTemplate{\mathrm{Pr}}{#1}[#2][#3]%
}

\NewDocumentCommand{\E}{m O{} o}%
{%
    \probabilisticFunctionTemplate{\mathrm{E}}{#1}[#2][#3]%
}

\NewDocumentCommand{\Var}{m O{} o}%
{%
    \probabilisticFunctionTemplate{\mathrm{Var}}{#1}[#2][#3]%
}

\DeclareDocumentCommand{\bigO}{m o}%
{%
    \functionTemplate{\mathrm{O}}{(}{)}{#1}[#2]%
}

\DeclareDocumentCommand{\smallO}{m o}%
{%
    \functionTemplate{\mathrm{o}}{(}{)}{#1}[#2]%
}

\DeclareDocumentCommand{\bigTheta}{m o}%
{%
    \functionTemplate{\upTheta}{(}{)}{#1}[#2]%
}

\DeclareDocumentCommand{\bigOmega}{m o}%
{%
    \functionTemplate{\upOmega}{(}{)}{#1}[#2]%
}

\DeclareDocumentCommand{\smallOmega}{m o}%
{%
    \functionTemplate{\upomega}{(}{)}{#1}[#2]%
}

\DeclareDocumentCommand{\eulerE}{o}%
{%
    \mathOrText{\mathrm{e}\IfNoValueTF{#1}{}{^{#1}}}%
}

\DeclareDocumentCommand{\poly}{m o}%
{%
    \functionTemplate{\mathrm{poly}}{(}{)}{#1}[#2]%
}

\createFunction{\id}{\mathrm{id}}

\NewDocumentCommand{\ind}{m o o}%
{%
    \IfNoValueTF{#2}%
    {%
        \mathOrText{\indicatorFunctionSymbol_{#1}}%
    }%
    {%
        \functionTemplate{\indicatorFunctionSymbol_{#1}}{(}{)}{#2}[#3]%
    }%
}

\DeclareDocumentCommand{\dom}{m o}%
{%
    \functionTemplate{\mathrm{dom}}{(}{)}{#1}[#2]%
}

\DeclareDocumentCommand{\rng}{m o}%
{%
    \functionTemplate{\mathrm{rng}}{(}{)}{#1}[#2]%
}

\DeclareDocumentCommand{\d}{o}%
{%
    \mathrm{d}\IfNoValueTF{#1}{}{^{#1}}%
}

\DeclareDocumentCommand{\set}{m m o}%
{%
    \mathOrText{\IfNoValueTF{#3}{\left}{#3}\{#1\ \IfNoValueTF{#3}{\left}{#3}\vert\ \vphantom{#1}#2\IfNoValueTF{#3}{\right.}{}\IfNoValueTF{#3}{\right}{#3}\}}%
}

%% file: packages_and_commands/additional_commands.tex
\DeclareDocumentCommand{\randomProcess}{m o}
{%
    \mathOrText{X^{(#1)}\IfNoValueTF{#2}{}{_{#2}}}%
}

\DeclareDocumentCommand{\transformedProcess}{o}
{%
    \mathOrText{Y\IfNoValueTF{#1}{}{_{#1}}}%
}

\DeclareDocumentCommand{\filtration}{o}
{%
    \mathOrText{\mathcal{F}\IfNoValueTF{#1}{}{_{#1}}}%
}

\newcommand*{\timePoint}{\mathOrText{t}}

\newcommand*{\hyperbolicExponent}{\mathOrText{\gamma}}
\newcommand*{\numberOfVertices}{\mathOrText{n}}

\newcommand*{\infectionRate}{\mathOrText{\lambda}}

\newcommand*{\infectionConstant}{\mathOrText{c}}

\newcommand*{\deimmunizationRate}{\mathOrText{\varrho}}
\newcommand*{\contactProcess}{\mathOrText{C}}

\newcommand*{\timeContinuous}[1]{\mathOrText{\tau_{#1}}}

\DeclareDocumentCommand{\lyapunovHelper}{o}
{%
    \mathOrText{f\IfNoValueTF{#1}{}{(#1)}}%
}

\DeclareDocumentCommand{\lyapunovFunction}{m o}
{%
    \mathOrText{F\IfNoValueTF{#2}{}{(#2)}}%
}

\DeclareDocumentCommand{\potentialFunctionIMinusR}{o}
{%
    \mathOrText{H\IfNoValueTF{#1}{}{(#1)}}%
}

\newcommand*{\infectedSet}[1]{\mathOrText{I_{#1}'}}
\newcommand*{\susceptibleSet}[1]{\mathOrText{S_{#1}'}}
\newcommand*{\recoveredSet}[1]{\mathOrText{R_{#1}'}}

\newcommand*{\conStar}[2]{\mathOrText{(#1,#2)}-connected-star}
\newcommand*{\Star}[2]{\mathOrText{(#1,#2)}-star}

%% file: content/introduction.tex
A great amount of processes in the real world follows similar core mechanics that are well modeled as a diffusion process on graphs. This includes information diffusion~\cite{SunCM2023SocialInfluenceMaximization,JiangRF2023PoliticalLearning,LiuRGCXTL2023DOVID19,SunRZLY2022HypergraphAttentionNetwork,RazaqueRKAR2022Survey,SharmaHSL2021FakeNews}, rumor spreading~\cite{kempe2003maximizing}, infections~\cite{Pastor-SatorrasCVMV15Survey,leskovec2007cost}, and computer viruses~\cite{berger2005spread,BorgsCGS10Antidote}. The core dynamics are well described by multiple processes defined in epidemiology~\cite{Pastor-SatorrasCVMV15Survey}. The most common processes are extensions of the so called \emph{SI process}. In this model, vertices are in one of two states: \emph{susceptible~(S)} or \emph{infected~(I)}. The states change following a continuous-time Markov chain on an underlying graph in which the infection randomly spreads with an \emph{infection rate} $\lambda$ across edges.

The SI process spawned several well-known extensions, such as the SIS process (or \emph{contact process}) and the SIRS process. The SIS process introduces reinfection, meaning that infected vertices become susceptible again at a rate commonly normalized to 1. The SIRS model also allows reinfection, but previously infected vertices enter a \emph{recovered (R)} state first, in which they are fully immune to the infection. They transition back into the susceptible state at a rate of $\varrho$, called the \emph{deimmunization rate}. For both the SIS and the SIRS process, the only absorbing state is the one in which every vertex is susceptible. The random time until the process (SIS or SIRS) reaches its absorbing state is called \emph{survival time} and is of particular interest, as it often exhibits a threshold behavior: in a very short interval of the values of $\lambda$ the expected survival time changes from being logarithmic in the graph size to super-polynomial. The region where this occurs is called the \emph{epidemic~threshold}.
An important line of research aims at understanding how properties of the host graph influence the existence and the placement of the epidemic threshold.

The epidemic threshold for the SIS process has been studied extensively both empirically \cite{PhysRevE.86.041125,ferreira2016collective} and theoretically \cite{NamNS22SISinfinite,BorgsCGS10Antidote,ganesh2005effect}. In particular, \textcite{ganesh2005effect} give an upper bound on the epidemic threshold based on the expansion properties of the host graph. Interestingly, this also includes graphs with very bad expansion. For example, on stars with $n$ leaves, the infection already survives super-polynomially long when for an arbitrary constant $\varepsilon > 0$, the infection rate $\lambda$ is of order $n^{-1/2 + \varepsilon}$. \textcite{berger2005spread} use this result in order to show super-polynomial survival times for preferential-attachment graphs, which are scale-free and therefore follow a power-law degree distribution. That is, they contain vertices of very high degree, and if the infection reaches these vertices, then it survives super-polynomially long in their neighborhood.

For the epidemic threshold for the SIRS process, there are many empirical results \cite[e.g.,][]{Wang_2017,PhysRevLett.86.2909,ferreira2016collective} as well as results that use simplifying assumptions such as mean-field approaches~\cite{prakash2012threshold,Bancal10}. On the theoretical side however, the results are quite sparse. \textcite{FrGoKlKrPa2024} give a bound for the epidemic threshold based on the expansion property of a (sub-)graph similar to the bound for the SIS process. However, they also show, that infections on stars never survive super-polynomially long when~$\varrho$ is assumed to be constant, regardless of the choice of~$\lambda$. Recently, \textcite{lam2024optimal} and \textcite{GoebelKKP25Resistance} independently showed that a polynomial expected survival time is achieved above the same threshold at which it becomes super-polynomial for the SIS process. The polynomial then keeps growing for larger $\lambda$ until it stops once $\lambda$ is constant. This implies that large star subgraphs alone are not sufficient in the SIRS model to show super-polynomial survival time thresholds.

In combination, these results suggest that recovered vertices have a big influence on graphs with very few paths between different vertices, whereas the impact of recovered vertices on graphs with many paths between each pair of vertices is far less pronounced.
However, since stars and expanders are structurally very far apart, it remains, up to date, unclear whether an epidemic threshold emerges for the SIRS process already for graphs with a far smaller degree of connectedness than present in expanders.

\subsection{Our contribution}

In this work, we identify a latent structure, namely \emph{connected stars}, of the host graph that results in an epidemic threshold for the SIRS process that requires far fewer connections between vertices than expanders. \emph{Connected stars} are multiple stars whose centers induce a connected component, e.g., a path.
We furthermore show that connected stars are present in many random-graph models that follow a power-law degree distribution.
Since such a distribution has been observed in a broad range of real-world networks~\cite{barabasi1999emergence}, this suggests that connected stars are present in many real-world graphs, effectively resulting in an epidemic threshold for the SIRS process similar to that for the SIS process.

In more detail, we show for graphs with~$n$ vertices that having $\lfloor\ln^2(n)\rfloor$ stars with $\bigOmega{n^c}$ leaves each and whose centers form a connected component results in a lower bound of the expected SIRS survival time of order $n^{\ln(n)}$ (\Cref{cor:survival}).
More generally, we derive a lower bound for the expected SIRS survival time on graphs that have a subgraph of~$k$ connected stars with~$\ell$ leaves each (\Cref{thm:survival}).

We note that our result from \Cref{cor:survival} is reasonably tight, as any constant number of connected stars still results, for constant deimmunization rates and arbitrary infection rates, in an at most polynomial expected SIRS survival time.
This follows from the upper bound of the SIRS survival time on single stars by \textcite{FrGoKlKrPa2024}, which uses that it only takes a polynomial time until the star center stays resistant for a longer time than it takes for all other vertices to lose their infection. This can be extended to show a polynomial upper bound for graphs with a constant-size vertex cover by bounding the time until vertices lose the infection while all of the vertices in the vertex cover stay resistant.

In \Cref{sec:real_world_graphs}, we show that many real-world graph models contain connected stars as subgraphs. The main property of the graph models that we use is that they are scale-free, meaning that their degree distribution follows a power-law. This ensures that the graphs contain many high-degree vertices that fill the role of the star centers. Note that in a power-law distribution, the $\ln^2(n)$ highest-degree vertices likely all have a degree that is only smaller than the largest degree by a poly-logarithmic factor. Hence, by considering multiple stars, the individual stars do not become substantially smaller.

In many scale-free graphs, the highest-degree vertices tend to be well connected. We make that explicit in \Cref{sec:con_star_in_real_graphs} for hyperbolic random graphs, geometric inhomogeneous random graphs, and Chung-Lu graphs with power-law kernel. Thus, \Cref{cor:survival} is directly applicable to these random graph models, yielding respective survival time bounds (\Cref{cor:scale_free_survival}). Similar results can be obtained on any scale-free graph as long as it can be shown that the highest degree vertices likely contain a connected subgraph.

Last, we compare our epidemic threshold for scale-free graphs from \Cref{cor:scale_free_survival} to a threshold by \textcite{FrGoKlKrPa2024} that is based on expander subgraphs. Either threshold depends on the power-law exponent $\gamma$ and the deimmunization rate $\varrho$ of the SIRS process. Generally, our threshold is better than the expansion threshold once $\gamma$ is sufficiently large, as then the central clique becomes very small. The exact point where it becomes better changes based on $\varrho$. Our threshold is better for larger values of $\varrho$ and becomes worse for small $\varrho$, while the expander threshold does not depend on $\varrho$ at all.

%% file: content/prelims.tex
We first introduce general notation before we define the SIRS process formally.
Then, we define important graph classes that we study.
Some mathematical tools that are important to our analysis can be found in \Cref{sec:mathTools}.
We note that our notation for SIRS processes follows \textcite{FrGoKlKrPa2024} in order to allow for an easier comparison of our results to theirs.

\paragraph{General notation and conventions.}
For all $a, b \in \N$, we define $[a .. b] = [a, b] \cap \N$.
Graphs in this work are simple and undirected, and
all of our big-O notation refers to asymptotics in the number of vertices~$n$.
We say that an event~$A$ occurs \emph{asymptotically almost surely} (a.a.s.) if and only if $\Pr{A} = 1 - \smallO{1}$. If $\Pr{\overline{A}}$ is super-polynomially small, we say it occurs \emph{with high probability} (w.h.p.).

A \emph{Poisson process of rate $r \in \R_{> 0}$} is a homogeneous Poisson point process that outputs almost surely a (random) countably infinite subset of~$\R_{\geq 0}$ where the smallest output as well as the difference between two neighboring outputs each follow an exponential distribution with parameter~$r$, independent of any other random events.

\paragraph{SIRS processes.}

An SIRS process runs on a graph $G=(V,E)$ and is parametrized by an \emph{infection rate} $\infectionRate \in \R_{>0}$ and a \emph{deimmunization rate} $\deimmunizationRate \in \R_{>0}$.
These two rates determine the speed at which vertices change from~S to~I and from~R to~S, respectively, noting that the rate from~I to~R is normalized to~$1$.

Each edge $e \in E$ has a corresponding Poisson process~$M_e$ of rate $\infectionRate$, and each vertex $v \in V$ has two corresponding Poisson processes $N_v$ and $O_v$ with rate $1$ and $\deimmunizationRate$, respectively. We call these processes \emph{clocks}, and when a time point $\timePoint \in \R_{\geq 0}$ is part of a clock's output, we say that the clock \emph{triggers} at~$\timePoint$. All of the clocks are independent. Note that, almost surely, no two clocks trigger at the same time. Let $\{\gamma_i\}_{i\in\N}$ with $\gamma_0=0$ denote the (random) sequence of all these triggers.

We define an SIRS process $(C_t)_{t \in \R_{\geq 0}}$ to be a process that partitions~$V$ for all time points $\timePoint \in \R_{\geq 0}$ into the set~$\susceptibleSet{\timePoint}$ of \emph{susceptible} vertices, the set~$\infectedSet{\timePoint}$ of \emph{infected} vertices, and the set~$\recoveredSet{\timePoint}$ of \emph{recovered} vertices. That is, $C_t = (\susceptibleSet{\timePoint}, \infectedSet{\timePoint}, \recoveredSet{\timePoint})$. The value of~$C_0$ is part of the input of the process. The other values are defined inductively and only change when clocks trigger. That is, for all $i \in \N$, the process~$C$ is constant on $[\gamma_i, \gamma_{i+1})$.
The transitions are defined for all $i \in \N$ and $s \in [\gamma_i, \gamma_{i + 1})$ as follows:

\begin{itemize}
      \item \textbf{Susceptible to infected.}
            Let $e = \{u, v\} \in E$ with $\gamma_{i + 1} \in M_e$ and $u \in \infectedSet{s}$ as well as $v \in \susceptibleSet{s}$.
            Then for all $t \in [\gamma_{i + 1}, \gamma_{i + 2})$, we have $u, v \in \infectedSet{t}$.
            We say that \emph{$u$ infects~$v$} (at time~$\gamma_{i + 1}$).
      \item \textbf{Infected to recovered.}
            Let $v \in V$ with $\gamma_{i + 1} \in N_v$ and $v \in \infectedSet{s}$.
            Then for all $t \in [\gamma_{i + 1}, \gamma_{i + 2})$, we have $v \in \recoveredSet{t}$.
            We say that \emph{$v$ recovers} (at time~$\gamma_{i + 1}$).
      \item \textbf{Recovered to susceptible.}
            Let $v \in V$ with $\gamma_{i + 1} \in O_v$ and $v \in \recoveredSet{s}$.
            Then for all $t \in [\gamma_{i + 1}, \gamma_{i + 2})$, we have $v \in \susceptibleSet{t}$.
            We say that \emph{$v$ becomes susceptible} or that it \emph{loses its immunity} (at time~$\gamma_{i + 1}$).
\end{itemize}

No transitions besides the ones above occur. Note that some triggers do not result in a change of state. Hence, in the remaining paper, we only consider times at which state changes occur. Formally, we consider $\{\gamma_0\} \cup \{\gamma_i \mid i \in \N_{\geq 1} \land \contactProcess_{\gamma_i} \neq \contactProcess_{\gamma_{i - 1}}\}$, which we index by the increasing sequence $\{\timeContinuous{i}\}_{i \in \N}$. For all $i \in \N$, we call $\timeContinuous{i}$ the $i$-th \emph{step} of~$C$.

The focus of our analysis is the \emph{survival time} $T \coloneqq \inf \{\timePoint \in \R_{\geq 0} \mid \infectedSet{\timePoint} = \emptyset\}$  of~$C$, and we say that the infection \emph{dies out} or \emph{goes extinct} at~$T$. When $T$ is in expectation super-polynomial in $n$, we say that the infection becomes \emph{epidemic}. The \emph{epidemic threshold} is the infimum of the regime for the infection rate~$\infectionRate$ in that the infection becomes epidemic.

In most lemmas and theorems, we consider a process starting at some given time $t \in \R_{\geq0}$. In all these cases, we consider~$t$ to be some stopping time with respect to the natural filtration of the SIRS process. That is,~$t$ is chosen without any knowledge of how the process behaves in the future. As the SIRS process changes based on Poisson processes, which are memoryless, the behavior also does not depend on the past and is only dependent on the state of the system at time~$t$.

\paragraph{Important graph classes.}
A \emph{star} is a graph over $m \in \N_{\geq 1}$ vertices with one distinct vertex (the \emph{center}) of degree $m - 1$, and all other vertices (the \emph{leaves}) with degree~$1$.

Given $k, \ell \in \N_{\geq 1}$, a concept central to our analysis is a graph that consists of exactly~$k$ stars, each with exactly~$\ell$ leaves.
We call such a graph a \emph{\Star{k}{\ell}}.
If~$k$ and~$\ell$ are not specified, we may also refer to such a graph as a \emph{disjointed star}.
Moreover, if the~$k$ centers induce a connected component, we call the graph a \emph{\conStar{k}{\ell}}, or just a \emph{connected star} if~$k$ and~$\ell$ are not specified.

We show in \Cref{sec:real_world_graphs} that disjointed stars are w.h.p. subgraphs in graphs known as \emph{random scale-free graphs}, whose degree distributions follow a power-law.
Formally, given a constant \emph{power-law exponent $\gamma \in \R_{> 0}$}, a random scale-free graph is a random graph~$G$ over $n \in \N_{\geq 1}$ vertices such that, w.h.p., for each degree $d \in [1 .. n]$, the fraction of vertices in~$G$ with degree~$d$ is proportional to $d^{-\gamma}$.
We note $\gamma >2$ implies that the average degree of~$G$ is constant, and we obtain the exact fraction of vertices with degree~$d$ by multiplying~$d^{-\gamma}$ with a constant $C\in \R_{>0}$.
However,~$C$ is not essential to our results as it is absorbed by the big-O notation, which is why we usually omit it.

We are interested in three common random scale-free graph models, which we show contain w.h.p. connected stars of a decent size.
We define these three models in the following in the level of detail that is sufficient for our purposes.
For more information, we refer to the references that we give.

\textbf{Hyperbolic random graphs}~\cite{krioukov2010hyperbolic} (HRGs) are generated by distributing points uniformly at random on a hyperbolic disk of radius $R \in \R_{> 0}$ and then connecting two vertices if and only if they are within distance~$R$ of each other. \textcite[Theorem~2.2]{gugelmann2012random} showed that the degree distribution of HRGs follows a power-law with exponent $\gamma \in \R_{> 0}$ that depends on the model parameters. The degree of vertices follows a binomial distribution around its expected degree, where vertices closer to the center have a higher degree \cite[Theorem~3.2]{gugelmann2012random}. All the vertices at distance at most~$R/2$ from the center form a clique. This distance implies, for some constant $c \in \R_{>0}$, an expected degree of at least $c\sqrt{n}$ while vertices further away have a smaller degree than that.

\textbf{Chung--Lu graphs}~\cite{chung2002average} are obtained by defining an expected-degree sequence and then connecting each pair of vertices independently with probability proportional to the product of their expected degrees~\cite{chung2002average}. This assures that the chosen expected degrees are indeed the expectation of the degree in the graph. The actual degree is binomially distributed with this mean. We are interested in the case where the expected-degree sequence follows a power-law with exponent $\gamma \in (2, 3)$. In the definition, the probability of two vertices being connected is capped at one, and for $\gamma \in (2, 3)$, two vertices with a product of their expected degrees that is bigger than the expected number of edges in the graph are connected with probability~$1$. In particular, there exists a constant $c \in \R_{>0}$ such that all vertices with expected degree at least $c \sqrt{n}$ are connected to each other.

\textbf{Geometric inhomogeneous random graphs}~\cite{bringmann2019geometric} (GIRGs) are a modification of Chung--Lu graphs that add a geometric component. Each vertex has a position in the $d$-dimensional Torus in addition to its expected degree. Vertices are more likely to connect to each other the closer they are. However, in regards to the properties considered in this paper, GRIGs behave the same as Chung--Lu graphs. The degree distribution of GIRGs follows a power-law with exponent~$\gamma$, typically assumed to be between~$2$ and~$3$~\cite[Theorem~2.1]{bringmann2019geometric}. Moreover, the degree of each vertex follows a binomial distribution around the mean based on their weight. Furthermore, vertices whose expected degrees multiply to more than the expected number of edges in the graph are by definition always connected. Hence, there again exists a constant $c \in \R_{>0}$ such that all vertices with expected degree at least $c \sqrt{n}$ are connected to each other.

%% file: content/SIRS_connected_stars.tex
Our main result is \Cref{thm:survival}, which proves a general lower bound of the SIRS survival time on $k \in \N_{\geq 1}$ connected stars.
The bound relies heavily on a probability $p \in (0, 1)$, which is the probability of the infection dying out in a single star in a time interval of constant length.
Whenever a star obtains sufficiently many infected vertices, we say it becomes \emph{active}. It becomes \emph{inactive} again when none of its vertices are infected.
If the infection rate $\lambda \in \smallOmega{kp}$, currently inactive stars become activated faster than active stars become inactive, leading to an expected survival time exponential in~$k$.

\begin{restatable}{theorem}{Survival}\label{thm:survival}
    Let $k, \ell \in \N_{\geq 1}$, and let~$G$ be a graph over $n \in \N_{\geq 1}$ vertices with a \conStar{k}{\ell} as a subgraph. Moreover, let $\varepsilon \in \R_{>0}$ be a constant, and let~$C$ be an SIRS process on~$G$ with infection rate $\lambda \in (0, 1]$ such that $\lambda \in \bigOmega{\ell^{-1/2+\varepsilon}}$ and with constant deimmunization rate $\varrho \in \R_{> 0}$ that starts with one of the star centers infected and all other vertices susceptible. Let~$T$ be the survival time of~$C$. Then there is a $p \in \bigTheta{(\infectionRate^2\ell)^{-(1-\varepsilon)\deimmunizationRate}}$ such that $\lambda \in \smallOmega{kp}$ implies that $\E{T} \in \bigOmega{(kp)^{-1}\left(\lambda/(k p)\right)^{k-1}}[\big]$.
\end{restatable}

Note that the condition $\lambda \in \smallOmega{kp}$ is only an implicit bound on $\lambda$ as $p$ depends on $\lambda$. Rearranging the condition to be explicit yields $\lambda \in \smallOmega{\left(\frac{k}{\ell^{(1-\varepsilon)\varrho}}\right)^{1/(1+2(1-\varepsilon)\varrho)}}$.

Furthermore, note that in \Cref{thm:survival}, the infection rate~$\lambda$ is upper-bounded by~$1$. The reason is that we use in the proof that in a single star, the number of infected leaves tends to go towards $\lambda n$, which is not possible for $\lambda>1$ as the number of infected leaves is upper bounded by $n$. However, throughout all the proofs, we only ever use~$\lambda$ as a lower bound for the infection rate. Hence, \Cref{thm:survival} is also applicable for all $\lambda>1$, resulting in the same (static) bound as for $\lambda=1$.

The survival time lower bound from \Cref{thm:survival} is exponential in~$k$ with $\frac{\lambda}{kp}$ in the base. Hence, if a graph with $n \in \N_{\geq 1}$ vertices contains a super-constant number of stars in the \Star{k}{\ell} and if $\lambda$ is bigger than $kp$ by a polynomial in $n$, the lower bound for the expected survival time is super-polynomial. This is in contrast to the results of \textcite{FrGoKlKrPa2024}, where the authors showed that on a single star, an SIRS infection never survives super-polynomially long when the deimmunization rate is constant.

\Cref{cor:survival} makes this super-polynomial lower bound explicit by choosing common values for $k$ and $\ell$ in \Cref{thm:survival}. We show in \Cref{sec:real_world_graphs} that many real-world graph models contain connected stars of the size chosen in \Cref{cor:survival}.

\begin{restatable}{corollary}{CorSurvival}\label{cor:survival}
    Let $n \in \N_{\geq 1}$, let $c\in \R_{>0}$ be a constant, and let~$G$ be a graph over~$n$ vertices that contains a \conStar{\lfloor\ln^2(n)\rfloor}{\bigOmega{n^c}} as subgraph. Moreover, let $\varepsilon \in \R_{>0}$ be a constant, and let~$C$ be an SIRS process on $G$ with constant deimmunization rate $\varrho \in \R_{> 0}$ and infection rate $\lambda \in (0, 1]$ such that $\lambda \in \bigOmega{n^{-\frac{c}{2+1/\deimmunizationRate}+\varepsilon}}[\big]$ that starts with one of the star centers infected and all other vertices susceptible. Let $T$ be the survival time of $C$. Then $\E{T} \in \bigOmega{n^{\ln(n)}}$.
\end{restatable}

Note that using similar calculations also give a survival time of order $n^{n^\varepsilon}$ when choosing $k=n^\varepsilon$. This increases the threshold for $\lambda$ slightly. However, in this paper we focus on the result in \Cref{cor:survival}.

The threshold for $\infectionRate$ depends on $\deimmunizationRate$, which it normally does not in previous results on this model. In \Cref{sec:real_world_graphs:comparison}, we compare how strong the threshold is on hyperbolic random graphs compared to previously known bounds.

In the following, we outline the proof of \Cref{thm:survival} and discuss the main lemmas that we use in order to prove it before we provide more detailed proof ideas for the main lemmas in \Cref{sec:theory:star-survival,sec:theory:activating-stars}. The full proofs of all of them are given in \Cref{ap:theory}.

\paragraph{Proof idea for \Cref{thm:survival}.}

We abstract the process $C$ into a meta process $A$ that only considers the \conStar{k}{\ell}, which we call~$G^\star$ in the following. Each star of~$G^\star$ is a vertex in $A$, and it is either active or inactive. Each star of~$G^\star$ becomes active once its center is infected and it infects enough of its leaves. The star becomes inactive again when no vertex in it is infected anymore. The corresponding vertex in~$A$ is active if and only if the star in $G^\star$ is active.

The meta process $A$ behaves roughly like an SIS process with parameters that depend on the parameters of~$C$ and the number of leaves in the stars. In our analysis, we focus on the number of active vertices in~$A$. Discretizing~$A$ with respect to the changes of the number of active vertices yields a random process $X \coloneqq (X_t)_{t\in \N}$ that increases or decreases by one in each step. Our goal is to show that this random variable dominates a biased gambler's ruin instance and to then use the expected survival time of this gambler's ruin instance (\Cref{pre:gamblersRuin}) as a lower bound for the original process. To this end, we bound the probability of~$X$ increasing and decreasing.

The process~$X$ decreases whenever one of the stars becomes inactive. In order to bound how often this happens, we show a lower bound on the distribution of how long the stars survive on their own (\Cref{lem:active_geom}). This bound is a generalization of the known lower bound of the expected survival time of single stars in the SIRS model by \textcite{gobel2025gradually}.
We discuss details in \Cref{sec:theory:star-survival}.

The process~$X$ increases whenever a star becomes active. This happens when the infected center of an active star infects the center of an inactive star and this center then infects enough leaves. As the star centers form a connected component, whenever there is at least one active and one inactive star, there are also centers of one active and one inactive star that are connected. In \Cref{lem:infect_probability} we lower-bound the probability of the inactive star being infected in a fixed time interval of constant length.
The details are given in \Cref{sec:theory:activating-stars}.

Last, in the proof of \Cref{thm:survival} in \Cref{ap:theory}, we combine the bounds on how long it takes to increase and decrease~$X$, lower-bounding the probability of an increase in each step. This implies a domination of the gambler's ruin instance and results in a lower bound on the expected survival time of~$X$. In order to transform this into a lower bound of the continuous meta process, we reverse the discretization by lower-bounding the expected time each decrease of~$X$ takes in the meta process. Noting that the meta process survives at most as long as the original process concludes the proof.

\paragraph{The meta process.}
The definition of the meta process is based on the notion of an \emph{active star}. A star becomes active once it has enough infected vertices and then stays active until it reaches a state with no infected vertices. We use this notion to abstract each star into a single meta vertex and then analyze the process on this meta graph.

\begin{definition}[Active star]
    Let $k, \ell \in \N_{\geq 1}$, and let~$G^\star$ be a \conStar{k}{\ell}.
    Moreover, let $\varepsilon \in \R_{>0}$ be a constant, and let~$C$ be an SIRS process on a super-graph of~$G^\star$ (over $n \in \N_{\geq 1}$ vertices) with infection rate $\lambda \in (0, 1]$ such that $\lambda \in \bigOmega{\ell^{-1/2+\varepsilon}}$ and with constant deimmunization rate $\varrho \in \R_{> 0}$. Furthermore, let $c=\frac{\deimmunizationRate}{4(2+\deimmunizationRate)}$, and let $d \in \R_{>0}$ be the largest constant such that $d \leq c/7$ and $\eulerE^{-2d/c}\geq 1-\varepsilon/2$. For each star $s$ of~$G^\star$ and each time $t\in \R_{\geq 0}$, let $I_{s,t}$ be the number of infected vertices of $s$ at time $t$. We say that $s$ is \emph{active} at time $t \in \R_{> 0}$ if and only if there exists a time $t' \in \R_{\leq t}$ such that $I_{s,t'} \geq \lambda d\ell$ and for all $t^* \in [t', t]$, we have $I_{s,t^*}>0$.
\end{definition}

Using this definition, we define the meta process and the discrete meta process of an SIRS infection on connected stars. They are projections of the original process on a smaller state space that only considers which stars are active at each time point.

\begin{definition}[Meta process]
    \label{def:meta-process}
    Let $k, \ell \in \N_{\geq 1}$, and let~$G^\star$ be a \conStar{k}{\ell}.
    Moreover, let $\varepsilon \in \R_{>0}$ be a constant, and let $C$ be an SIRS process on a super-graph of~$G^\star$ (over $n \in \N_{\geq 1}$ vertices) with infection rate $\lambda \in (0, 1]$ such that $\lambda \in \bigOmega{\ell^{-1/2+\varepsilon}}$ and with constant deimmunization rate $\varrho \in \R_{> 0}$. For each star $s$ of~$G^\star$ and each $t \in \R_{\geq 0}$, let $A_{s,t}$ be the indicator random variable of the event that~$s$ is active at time $t$. We call the union of these indicator variables over all~$k$ stars the \emph{meta process} $A$ of $C$. Furthermore, for each $t \in \R_{\geq 0}$, let~$X_t$ be the random variable of the number of active stars at time~$t$, and for each $i \in \N$, let $\tau(i)$ be the $i$-th time point at which $(X_t)_{t \in \R_{\geq 0}} \eqqcolon X$ changes its value, defining $\tau(0) = 0$. We call~$X$ the \emph{simplified meta process} of $C$, and $\left(X_{\tau(i)}\right)_{i \in \N}$ the \emph{discrete version} of $X_t$.
\end{definition}

\subsection{Survival Time of Stars}
\label{sec:theory:star-survival}

We analyze how long a star stays active in the meta process.
To this end, we adapt \cite[Theorem~4.9]{gobel2025gradually}\footnote{The conference version \cite{GoebelKKP25Resistance} does not contain the proofs of these statements, which is why we refer to \cite{gobel2025gradually} instead.}\!\!, which proves a lower bound on how long stars survive in the SIRS model in expectation.
Our resulting \Cref{lem:active_geom} extends this result such that we obtain a bound on the probability to survive for a certain period of time.

The proof of \cite[Theorem~4.9]{gobel2025gradually} splits the time until extinction into phases and shows that the number of these phases until the infection dies out dominates a geometric random variable. However, these phases can be quite short, resulting in weak concentration. In order to get a better concentration bound, we adjust the definition of a phase slightly.
Roughly, \textcite{gobel2025gradually} define a phase to be the time between the number of infected leaves dropping below a certain threshold and then going above it again (or the infection dying out)~\cite[Lemma~4.8]{gobel2025gradually}.
We augment this definition by requiring for each phase that the star center recovers after the number of leaves reaches the threshold.
This extra waiting time leads to stronger concentration results.

Our adjusted \Cref{lem:active_geom} is given below.
In \Cref{ap:theory} we discuss in more detail how to adjust the proof by \textcite{gobel2025gradually} in order for it to apply to our result.
We note that we use $\ell \in \N_{\geq 1}$ for the number of leaves below instead of $n \in \N_{\geq 1}$, as chosen by \textcite{gobel2025gradually}.

\begin{restatable}{lemma}{activeGeom}[{\cite[Theorem~4.9]{gobel2025gradually}}]\label{lem:active_geom}
    Let $G^\star$ be a star with $\ell \in \N_{\geq 1}$ leaves, and let $\varepsilon \in \R_{>0}$ be a constant. Let $C$ be an SIRS process on a super-graph of $G^\star$ with infection rate $\lambda \in (0, 1]$ such that $\lambda \in \smallOmega{\ell^{-1/2}}$ and with constant deimmunization rate $\varrho \in \R_{> 0}$. Moreover, let $t\in \R_{>0}$ be a time at which~$G^\star$ becomes active. Then there is a $p \in \bigTheta{(\infectionRate^2\ell)^{-(1-\varepsilon)\deimmunizationRate}}$ such that the expected time starting from $t$ until all of the vertices in $G^\star$ are not infected dominates a random variable $X \sim \mathrm{Geom}(p)-1$.
\end{restatable}

\subsection{Activating New Stars}
\label{sec:theory:activating-stars}

We analyze the probability that an active star $S_1$ activates an adjacent inactive star~$S_2$ within a time frame of constant length. To this end, we split this event into multiple events occurring in order, and we lower-bound each of their probabilities. The first event is the loss of immunity of the center of~$S_2$ within a time frame of length $1/\varrho$. This happens with constant probability by the definition of deimmunization clocks. The second event is the loss of immunity of the center of~$S_1$ within a time frame of length $1/\varrho$ after the first event, which happens with constant probability for the same reason as before. The third event is the infection of the center of~$S_1$ within a time frame of length $1/(1+\lambda)$ after event 2. This happens with constant probability as it is unlikely for the infection in $S_1$ to die out and conditional on the center of $S_1$ being infected again, the time until that happens is dominated by an exponential distribution with rate $1+\lambda$. The fourth event is the transmission of the infection from the center of $S_1$ to the center of $S_2$ within a time frame of length $1/(1+\lambda)$ after event 3. This happens with a probability in $\bigTheta{\lambda}$ as it requires the infection clock on the edge to trigger before the center of $S_1$ heals. The last event is the activation of $S_2$ by infecting enough leaves in a time frame of length $1$ after event 4. By \Cref{pre:center_infected} this happens with constant probability. Combining all of the events yields \Cref{lem:infect_probability}.

\begin{restatable}{lemma}{infectProbability}\label{lem:infect_probability}
    Let $G$ be a graph over $n \in \N_{\geq 1}$ vertices that contains two stars, $S_1$ and $S_2$, with $\ell \in \N_{\geq 1}$ leaves each and whose centers are connected. Let $\varepsilon \in \R_{> 0}$ be a constant, and let~$C$ be an SIRS process on~$G$ with infection rate $\lambda \in (0, 1]$ such that $\lambda \in \bigOmega{\ell^{-1/2+\varepsilon}}$ and with constant deimmunization rate $\varrho \in \R_{> 0}$. Let $T$ be a stopping time with respect to the meta process $A$ of $C$ such that $S_1$ is active and $S_2$ inactive at~$T$. The probability of the event~$E$ that $S_1$ infects $S_2$ within a time interval $I$ of length $2(1 + 1/\varrho + 1/(1+\lambda))$ starting from~$T$ is in $\bigOmega{\lambda}$.
\end{restatable}

%% file: content/application.tex
We show in \Cref{cor:scale_free_survival} that many common scale-free random graph models exhibit an at least super-polynomial expected survival time when an SIRS process is run on them.
In particular, we show that this is the case for hyperbolic random graphs~\cite{krioukov2010hyperbolic} (HRGs), geometric inhomogeneous random graphs~\cite{bringmann2019geometric} (GIRGs), and Chung--Lu graphs~\cite{chung2002average}, in all cases when the power-law exponent of their degree distribution is strictly between~$2$ and~$3$, which is commonly observed in real-world graphs~\cite{barabasi1999emergence}.

\begin{corollary}\label{cor:scale_free_survival}
    Let $\varepsilon \in \R_{>0}$ be a constant, and let~$G$ be a hyperbolic random graph, geometric inhomogeneous random graph, or Chung--Lu graph with $n\in \N_{\geq 1}$ vertices whose degree distribution follows a power-law with exponent $\gamma \in (2,3)$. Let $C$ be an SIRS process on $G$ with constant deimmunization rate $\varrho \in \R_{> 0}$ and infection rate $\lambda \in (0, 1]$ such that $\lambda \in \bigOmega{n^{-\frac{1}{(2+1/\deimmunizationRate)(\gamma-1)}+\varepsilon}}$ that starts with the highest-degree vertex infected and all other vertices susceptible. Let $T$ be the survival time of $C$. Then $\E{T} \in \bigOmega{n^{\ln(n)}}$.
\end{corollary}

We prove \Cref{cor:scale_free_survival} by applying \Cref{cor:survival} and showing that all three graph models exhibit sufficiently large connected stars.
In more detail, we first show in \Cref{lem:scale_free_degree} that random scale-free graphs with power-law exponent $\gamma \in \R_{> 2}$ contain w.h.p. $\ln(n)^2$ vertices of degree at least $\left(c n/(\ln(n)^2)\right)^{1/(\gamma-1)}$, which follows directly from the definition of the power-law degree distribution.

Next, we show in \Cref{lem:degree_star} that containing $k \in \N_{\geq 1}$ vertices with degree of at least $d \in \N_{\geq k}$ each is sufficient for a graph to contain a \Star{k}{\lfloor d/k \rfloor-1}. Such a star is obtained by greedily assigning $\lfloor d/k \rfloor-1$ leaves to each of the~$k$ centers without overlapping the set of leaves.

Last, we show in \Cref{lem:scale_free_conStar} that w.h.p. the high-degree vertices of HRGs, GIRGs, and Chung--Lu graphs with power-law kernel form a connected component, which concludes our result.

In the following, we show in \Cref{sec:real_world_graphs:existence-of-connected-stars} that random scale-free graphs have w.h.p. disjointed stars with a sufficiently high degree.
Afterward, we show in \Cref{sec:con_star_in_real_graphs} that these stars are w.h.p. connected in the random graph models of our interest.

Last, in \Cref{sec:real_world_graphs:comparison}, we compare our bounds from \Cref{cor:scale_free_survival} to the results of \textcite{FrGoKlKrPa2024}, who proved expected super-polynomial bounds for SIRS processes on graphs with expanding subgraphs.
In particular, we discuss when our upper bound on the infection threshold is lower (that is, more permissive) than theirs.

\subsection{Disjointed Stars in Scale-free Graphs}
\label{sec:real_world_graphs:existence-of-connected-stars}

We show that scale-free graphs contain w.h.p. a \Star{k}{\ell} for $k=\ln(n)^2$ and $\ell \in \bigTheta{\left( n/\ln(n)^{2\gamma}\right)^{1/(\gamma-1)}}[\big]$. We start by showing that, for some constant $c \in \R_{>0}$, scale-free graphs contain w.h.p. $\lfloor\ln(n)^2\rfloor$ vertices of degree at least $\left(c n/\ln(n)^2\right)^{1/(\gamma-1)}$.

\begin{lemma}\label{lem:scale_free_degree}
    Let $G$ be a random scale-free graph with $n\in \N_{\geq 1}$ vertices that follows a power-law degree distribution with power-law exponent $\gamma \in \R_{>2}$. Then there exists a constant $c\in \R_{>0}$ such that w.h.p. $G$ contains $\lfloor\ln(n)^2\rfloor$ vertices with a degree of at least $\left(c n/\ln(n)^2\right)^{1/(\gamma-1)}$ each.
\end{lemma}

\begin{proof}
    This result follows directly from the definition of a power-law degree distribution. As the degree distribution of $G$ follows a power-law with exponent $\gamma \in\R_{>2}$, the sum of all degrees of vertices in the graph is linear in~$n$. Hence, for all $d \in [1 .. n]$, by summing the number of vertices with degree $i \in [d .. n]$, there is a constant $c\in \R_{>0}$ such that the number of vertices of degree at least $d$ is w.h.p. at least $F(d) \coloneqq cn d^{-(\gamma-1)}$. Choosing $d' = \left(c n/\lfloor\ln(n)^2\rfloor\right)^{1/(\gamma-1)}$ yields that w.h.p. there are at least $F(d')= cn (d')^{-(\gamma-1)} = cn\frac{\lfloor\ln(n)^2\rfloor}{cn} =\lfloor\ln(n)^2\rfloor$ vertices of degree at least~$d'$ in~$G$.
\end{proof}

Next, we show that a graph $G$ with $k \in \N_{\geq 1}$ vertices of degree at least $d \in \N_{\geq k}$ each contains a \Star{k}{\lfloor d/k\rfloor-1}. It is obtained by greedily assigning each of the~$k$ high-degree vertices $\lfloor d/k\rfloor -1$ of their neighbors as leaves.

\begin{lemma}\label{lem:degree_star}
    Let $k \in \N_{\geq 1}$ and $d \in \N_{\geq k}$, and let $G$ be a graph that contains $k$ vertices, each of degree at least $d$. Then $G$ contains a \Star{k}{\lfloor d/k\rfloor-1} as a subgraph.
\end{lemma}

\begin{proof}
    Let $(v_i)_{i \in [1 .. k]}$ be vertices with degree at least $d$ in $G$. We build a subgraph~$G'$ of~$G$ by iteratively adding stars to it. In iteration $i \in [1 .. k]$, we add vertex $v_i$ to $G'$ and a set~$N_i$ of $\lfloor d/k\rfloor-1$ of its neighbors that are not in the set $S_i \coloneqq (\bigcup_{j \in [1 .. i - 1]}{N_j}) \cup \bigcup_{j \in [1 .. k]}{\{v_j\}}$. We also add the edges between $v_i$ and these neighbors. This is always possible as there are exactly $k$ vertices in $\bigcup_{j \in [1 .. k]}{\{v_j\}}$, and $(i-1)(\lfloor d/k\rfloor-1)$ vertices in $\bigcup_{j \in [1 .. i - 1]}{N_j}$. As~$v_i$ has at least~$d$ neighbors, it has at least $d-k-(i-1)(\lfloor d/k\rfloor-1)\geq d-k-(k-1)(\lfloor d/k\rfloor-1) \geq \lfloor d/k\rfloor-1$ neighbors outside of $S_i$.

    After all $k$ iterations, $G'$ is a \Star{k}{\lfloor d/k\rfloor-1}.
\end{proof}

Combining \Cref{lem:scale_free_degree,lem:degree_star} gives us that random scale-free graphs contain a \Star{k}{\ell} that is suitably large for \Cref{cor:scale_free_survival} to be applied.

\begin{corollary}\label{cor:scale_free_star}
    Let $G$ be a random scale-free graph with $n\in \N_{\geq 1}$ vertices which follows a power-law degree distribution with power-law exponent $\gamma \in \R_{>2}$. Then there exists a constant $c\in \R_{>0}$ such that for $k=\lfloor\ln(n)^2\rfloor$ and $\ell=\left(c n/\ln(n)^{2\gamma}\right)^{1/(\gamma-1)}$, w.h.p.~$G$ contains a \Star{k}{\ell} as a subgraph.
\end{corollary}

\subsection{Connected Stars in HRGs, GIRGs, and Chung--Lu Graphs}\label{sec:con_star_in_real_graphs}

We apply \Cref{cor:survival} to HRGs, GIRGs, and Chung--Lu graphs with power-law kernel. Using \Cref{cor:scale_free_star}, we see that all of these models contain w.h.p. large disjointed stars. In order to apply \Cref{cor:survival}, it remains to show that these stars are w.h.p. connected. To this end, we argue that the centers of the disjointed stars from \Cref{cor:scale_free_star} all lie w.h.p. within the central cliques that all of the discussed models contain. Hence, they are connected. We argue these results without diving too deep into the definitions of the models. We only use the properties of the models discussed in the \Cref{sec:preliminaries}.

\begin{lemma}\label{lem:scale_free_conStar}
    Let $G$ be a hyperbolic random graph, geometric inhomogeneous random graph, or Chung--Lu graph with $n\in \N_{\geq 1}$ vertices whose degree distribution follows a power-law with exponent $\gamma \in (2,3)$. Then there exists a constant $c\in \R_{>0}$ such that for $k=\lfloor\ln(n)^2\rfloor$ and $\ell=\left(c n/\ln(n)^{2\gamma}\right)^{1/(\gamma-1)}$, w.h.p. $G$ contains a \conStar{k}{\ell} as a subgraph.
\end{lemma}

\begin{proof}
    By \Cref{cor:scale_free_star}, there exists a constant $c\in \R_{>0}$ such that for $k=\lfloor\ln(n)^2\rfloor$ and $\ell=\left(c n/\ln(n)^{2\gamma}\right)^{1/(\gamma-1)}$, w.h.p. $G$ contains a \Star{k}{\ell} as a subgraph. Note that as $\gamma <3$, there exists a constant $\varepsilon \in \R_{>0}$ such that $\ell \in \bigOmega{n^{1/2+\varepsilon}}$.

    Hyperbolic random graphs, geometric inhomogeneous random graphs, and Chung--Lu graphs all contain central cliques. Vertices outside of this clique have, for some constant $c' \in \R_{>0}$, an expected degree of at most $c'\sqrt{n}$. As the actual degree is binomially distributed when given the expected value and as $\ell \in \bigOmega{n^{1/2+\varepsilon}}$, by a Chernoff bound and a union bound over all vertices, w.h.p. there is no vertex outside of the central clique with degree at least~$\ell$. Hence, w.h.p. all of the centers of the \Star{k}{\ell} are in the central clique and therefore connected.
\end{proof}

Combining \Cref{lem:scale_free_conStar} with \Cref{cor:survival} immediately implies \Cref{cor:scale_free_survival}, giving us a bound for the epidemic threshold on HRGs, GIRGs, and Chung--Lu graphs.

\subsection{Comparison of the Epidemic Threshold Between Expander Graphs and Hyperbolic Random Graphs}
\label{sec:real_world_graphs:comparison}

\Cref{cor:scale_free_survival} states a bound for the epidemic threshold on hyperbolic random graphs, based on connected stars. \textcite[Corollary~1.6]{FrGoKlKrPa2024} also state a bound on the epidemic threshold for hyperbolic random graphs, but it is based on the expansion properties of the central clique.
For some constant $c \in \R_{>0}$, \textcite[Corollary~1.6]{FrGoKlKrPa2024} show an upper bound on the threshold of $\infectionRate \geq \infectionConstant \numberOfVertices^{(\hyperbolicExponent-3)/2}$.

In the following, we discuss for which power-law exponents which of the two results leads to a smaller (and thus more permisse) upper bound on the epidemic threshold (\Cref{lem:compare}) and how this relates to the two different properties that are at the foundation of these two results. The proof of \Cref{lem:compare} can be found in \Cref{ap:real_world}

\begin{restatable}{lemma}{compare}\label{lem:compare}
    Let $G$ be a hyperbolic random graph with power-law degree exponent $\gamma \in (2, 3)$. Then the epidemic threshold upper bound from the connected stars (\Cref{cor:scale_free_survival}) is smaller than the expander bound from \cite[Corollary~1.6]{FrGoKlKrPa2024} if and only if $\gamma > 2 + (2\varrho+1)^{-1/2}$.
\end{restatable}

The value~$\gamma$ from \Cref{lem:compare} depends on~$\varrho$, which is a model parameter, and thus leads to different regimes in which which result yields a more permissive threshold.
For the sake of convenience, we call this value of~$\gamma$ (dependent on~$\varrho$) the \emph{meta threshold}.
In general, this meta threshold lies in between~$2$ and~$3$, with \Cref{cor:scale_free_survival} being more permissive when~$\gamma$ is above the threshold and \cite[Corollary~1.6]{FrGoKlKrPa2024} being more permissive otherwise.

More specifically, the bound from \Cref{cor:scale_free_survival} is more permissive when $\gamma$ gets close enough to~$3$. The reason is that when the power-law exponent approaches~$3$, the number of vertices in the central clique of the hyperbolic random graph~$G$ approaches~$0$. As \cite[Corollary~1.6]{FrGoKlKrPa2024} makes use of the size of this central clique for the bound, its threshold becomes arbitrarily bad. In contrast, \Cref{cor:scale_free_survival} only needs a logarithmic number of high degree vertices, which is w.h.p. there as long as $\gamma<3$. The degree of these high-degree vertices converges to $n^{1/2}$, hence the threshold of \Cref{cor:scale_free_survival} does not become arbitrarily bad.

On the other hand, when $\gamma$ gets close enough to~$2$, the bound from \cite[Corollary~1.6]{FrGoKlKrPa2024} becomes more permissive. The reason is that the size of the central clique approaches $n^{1/2}$ which means that the threshold from \cite[Corollary~1.6]{FrGoKlKrPa2024} approaches $n^{-1/2}$. In comparison, the threshold from \Cref{cor:scale_free_survival} is never as good as $n^{-1/2}$, as it is lower-bounded by $n^{-1/(2+1/\varrho)}$.

The threshold from \cite[Corollary~1.6]{FrGoKlKrPa2024} does not depend on $\varrho$ at all while the bound from \Cref{cor:scale_free_survival} becomes better the higher $\varrho$ is. For~$\varrho$ approaching~$0$, that is, vertices remain likely recovered for a long time, the meta threshold approaches~$3$ as the survival time of single stars becomes arbitrarily small, effecting the bound of \Cref{cor:scale_free_survival}. For~$\varrho$ approaching infinity, the meta threshold approaches~$2$, which means that \Cref{cor:scale_free_survival} yields a far more permissive regime than \cite[Corollary~1.6]{FrGoKlKrPa2024}.

Last, we note that, for~$\varrho$ approaching infinity, the SIRS model approaches the SIS model, in which infected vertices become directly susceptible again, skipping immunity.
In the actual SIS model, for all constants $\varepsilon \in \R_{> 0}$, a single star results already in an epidemic threshold of $\lambda \geq n^{-1/2 + \varepsilon}$~\cite{ganesh2005effect}.
This is also the value that the threshold in \Cref{cor:survival} approaches.

%% file: content/outlook.tex
While our results cover a broad range of relevant network models, we leave an extended analysis for the SIRS process of preferential-attachment graphs for future work. On such networks, it would be particularly interesting to compare the behavior of the SIRS process with that of SIS, given that stars are the key substructure that the analysis of \textcite{berger2005spread} is based on, together with the stark disparity in the behavior of the two processes on stars.
We note that our results in \Cref{sec:real_world_graphs} consider graphs with a central clique, which preferential-attachment graphs do not feature.
Hence, additional insights are required for a mathematical analysis.

Another direction is to investigate the bound for how fast a single star among the connected ones gets infected.
With our current bound, our epidemic threshold in \Cref{thm:survival} scales polynomially in the number~$k$ of stars, i.e., it becomes less permissive the higher~$k$ is.
However, as shown in \Cref{cor:survival} and its proof, this effect does not change the epidemic threshold asymptotically as long as~$k$ is sub-polynomial.
A more fine-grained analysis may be capable to remove this dependency on~$k$ altogether.

We furthermore believe that the epidemic threshold can be improved by incorporating the fact that the connected stars in scale-free graph models typically form a clique.
A detailed analysis using this assumption may lead to an epidemic threshold that actually becomes more permissive the more connected stars the graph contains.

Another interesting point for future research is to consider connected stars that may share some of their leaves and are not fully disjoint from one another.
In many scale-free graph models, there is even a substantial overlap of the neighborhood of high-degree vertices.
We believe that incorporating such an overlap into the analysis should result in even more permissive bounds on the epidemic threshold.

Finally, as most previous works, we assume that the deimmunization rate~$\varrho$ of the SIRS process is constant. It is an interesting open problem what happens when~$\varrho$ depends on~$n$.

%% file: content/acknowledgments.tex
This work received support by the program \emph{Research Program for International Talents~2025} from Ecole Polytechnique.

Last, the authors thank Janosch Ruff for valuable discussions about scale-free graph models.

%% file: content/appendix.tex
\section{Probabilistic tools}\label{sec:mathTools}

We bound the survival time of an SIRS process by showing that certain of its features stochastically dominate other random processes that are easier to analyze. We say that a random process $(X_t)_{t \in \R_{\geq 0}}$ \emph{dominates} another random process $(Y_t)_{t \in \R_{\geq 0}}$ if and only if there exists a coupling $(X'_t, Y'_t)_{t \in \R_{\geq 0}}$ between~$X$ and~$Y$ such that for all $t \in \R_{\geq 0}$, we have $X'_t \geq Y'_t$.

We use the following well-known bounds on the gambler's ruin process in our proofs in order to bound the time until specific events occur.

\begin{theorem}[Gambler's ruin~{\protect\cite[page~345]{feller1957introduction}}]\label{pre:gamblersRuin}
    In the gambler's ruin game, a player starts with some amount of money $P_0 \in \N_{\geq 1}$ and then repeatedly throws a coin that shows head with probability $p \in (0,1)$, independent of all other outcomes. If it shows heads, the player's money increases by~$1$, otherwise, it decreases by~$1$. The game ends at time~$T$ when the player either reaches the lower bound $\ell \in \N_{< P_0}$ or the upper bound $u \in \N_{> P_0}$ of money.
    Let $(P_t)_{t \in \N}$ be the amount of money that a player has in a gambler's ruin game that has a probability of $p \in (0,1)\setminus\{1/2\}$ for them to win in each step. Let $q=1-p$.  Then
    \begin{enumerate}
        \item $\Pr{P_T = \ell} = \frac{1-(p/q)^{u-P_0}}{1-(p/q)^{u-\ell}}$ and
        \item $\Pr{P_T = u} = \frac{1-(q/p)^{P_0-l}}{1-(q/p)^{u-\ell}}$.\qedhere
    \end{enumerate}
\end{theorem}

The following lemma shows how likely it is that an SIRS infection of a star infects sufficiently many leaves before the center recovers.

\begin{lemma}[{\cite[Lemma~4.5]{gobel2025gradually}}]
    \label{pre:center_infected}
    Let $G^\star$ be a star with $\ell \in \N_{\geq 1}$ leaves.
    Let~$C$ be an SIRS process on a super-graph of $G^\star$ with infection rate $\lambda \in (0, 1]$ such that $\lambda \in \smallOmega{\ell^{-1}}$ and with constant deimmunization rate $\varrho \in \R_{> 0}$.
    Moreover, let $t\in \R_{>0}$ be a time at which the center of~$G^\star$ is infected.
    Last, let $\varepsilon \in \R_{>0}$ be a constant, $c = \frac{\deimmunizationRate}{4(2 + \deimmunizationRate)}$, and $d \in \R_{> 0}$ be a constant such that $d \leq c/7$ and $\eulerE[-2d/c] \geq 1 - \varepsilon/2$.
    Assume that there are always at least~$2c\ell$ vertices that are not recovered during the considered time interval.
    Then starting from~$t$, the probability that~$C$ reaches a state with at least~$\infectionRate d \ell$ infected vertices before the center heals is at least $1 - \varepsilon$ for sufficiently large~$\ell$.
\end{lemma}

\section{Lower bound of the SIRS Survival Time on Connected Stars}\label{ap:theory}

In this section we rigorously proof the lemmas from \Cref{sec:theory}. We start by proving \Cref{cor:survival} by plugging in $k=\lfloor \ln(n)^2\rfloor$ and $\ell = \lceil n^c\rceil$ into \Cref{thm:survival}.

\CorSurvival*

\begin{proof}
    For a cleaner proof, we ignore constants introduced from big-O notation, in particular, all following inequalities are only true up to constant factors. For some constant $\varepsilon'\in \R_{>0}$ we choose later, we aim to apply \Cref{thm:survival} with $k= \lfloor\ln(n)^2\rfloor$, $\ell= \lceil n^c\rceil$, and where we rename the~$\varepsilon$ from \Cref{thm:survival} to~$\varepsilon'$. From now,~$\varepsilon$ refers exclusively to the variable of the same name from \Cref{cor:survival}.

    First note that $\frac{1}{2+1/\varrho} < \frac{1}{2}$, hence $\lambda \in \bigOmega{n^{-\frac{c}{2+1/\deimmunizationRate}+\varepsilon}}[\big]$ implies $\lambda \in \bigOmega{l^{-\frac{1}{2}+\varepsilon'}}[\big]$ when $\varepsilon'<\frac{1}{2}-\frac{1}{2+1/\varrho}$.

    We now show that choosing $\varepsilon'$ small enough implies $\lambda \in \smallOmega{kp}$. To this end, we lower-bound $\frac{\lambda}{kp}$ by a term in $\smallOmega{1}$.
    \begin{align*}
        \frac{\lambda}{kp} & \geq \frac{\lambda}{\ln(n)^2(\lambda^2\ell)^{-(1-\varepsilon')\varrho}}                                 \\
                           & = \frac{1}{\ln(n)^2}\lambda^{1+2(1-\varepsilon')\varrho}\ell^{(1-\varepsilon')\varrho}                  \\
                           & \geq \frac{1}{\ln(n)^2}\lambda^{1+2\varrho}\ell^{(1-\varepsilon')\varrho}                               \\
                           & \geq \frac{1}{\ln(n)^2}n^{(-\frac{c}{2+1/\varrho}+\varepsilon)(1+2\varrho)}n^{c(1-\varepsilon')\varrho} \\
                           & \geq \frac{1}{\ln(n)^2} n^{\varepsilon(1+2\varrho)-\varepsilon'c\varrho}.
    \end{align*}

    Note that choosing $\varepsilon'$ small enough in order to make the exponent positive gives us a constant $\varepsilon^*\in \R_{>0}$ such that $\frac{\lambda}{kp}\geq n^{\varepsilon^*}$. Hence, $\lambda \in \smallOmega{kp}$. Applying \Cref{thm:survival} gives us for sufficiently large $n$.
    \begin{align*}
        \E{T} & \geq (kp)^{-1}\left(\lambda/(k p)\right)^{k-1}     \\
              & \geq\frac{n^{\varepsilon^*(\ln(n)^2-1)}}{\ln(n)^2} \\
              & \geq n^{\ln(n)}.\qedhere
    \end{align*}
\end{proof}

Next we show \Cref{lem:active_geom}, which is a strengthened version of \cite[Lemma~4.9]{gobel2025gradually}. To this end, we modify the lemmas they used slightly in order for them to use on a slightly different definition of phases. This makes it possible to get a bound on the distribution of the time until an infection in a star dies out instead of just getting a bound on the expected time.

\activeGeom*

\begin{proof}
    Let $d \in \R_{> 0}$ be the constant defined in \cite[Lemma~4.8]{gobel2025gradually}. Our proof considers separate phases of~$C$.
    We define a phase to start whenever no previous phase is still ongoing and when the number of infected leaves (of~$G^\star$) is below $\lambda d \ell$. A phase ends at the first time at which the center recovers after the number of infected leaves exceeded $\lambda d \ell$ after the start of the phase.

    In the following, we aim at applying \cite[Lemma~4.8]{gobel2025gradually}, which implies that the number of phases dominates a geometric random variable with appropriate parameter.
    However, since our definition of a phase differs slightly from that of \textcite{gobel2025gradually}, we argue below why the lemma is still applicable.
    Afterward, we show that each phase has a constant probability to last for a time of at least~$1$. We then combine these two results, yielding a bound on the total length of all phases.

    The proof of \cite[Lemma~4.8]{gobel2025gradually} merges each center-susceptible phase and the following center-infected phase into a single step. It then bounds the number of leaves that recover in these phases by \cite[Corollary~4.7]{gobel2025gradually} and analyzes the probability of the remaining infected leaves recovering in center-recovered phases before a center-infected phase leads to the number of infected leaves exceeding $\lambda d \ell$.
    The only difference of our phase definition to theirs is that the start of a phase in our sense may be delayed by a center-infected phase, as our phase only ends after the center recovers. However, \cite[Corollary~4.7]{gobel2025gradually} already bounds the number of leaves that recover in center-infected phases, and adding one additional center-infected phase does not change the asymptotic statement of \cite[Corollary~4.7]{gobel2025gradually}. Hence, the exact same argumentation is sufficient in order to show that \cite[Lemma~4.8]{gobel2025gradually} is also applicable to our phases.

    Let $c = \frac{\varrho}{4(2 + \varrho)}$.
    We continue by bounding the probability of a phase to be the last.
    \cite[Lemma~4.8]{gobel2025gradually} states that this probability is at most $q \in \bigTheta{(\infectionRate^2\numberOfVertices)^{-(1-\varepsilon)\deimmunizationRate}}$.
    However, this requires to condition on the event~$E$ that there are at least $2c\ell$ vertices that are not recovered throughout the entire phase. By \cite[Lemma~4.4]{gobel2025gradually}, the probability that~$\overline{E}$ occurs during a phase is exponentially small (in~$\ell$). Hence, the probability that a phase is the last or that~$\overline{E}$ occurs during this phase is at most~$2q$.

    Last, we bound the length of all phases combined.
    Let $Y$ be the number of phases before~$G^\star$ becomes inactive.
    Therefore,~$Y$ dominates a geometric random variable with a parameter $2q$. As each phase includes the center recovering, its length dominates an exponentially distributed random variable with rate~$1$. Hence, each phase has a probability of at least $\eulerE^{-1}$ to last for a time of at least~$1$, independent of the length of all the other phases. Hence the total time until none of the vertices in $G^\star$ are infected dominates the sum of $Y$ random lengths that each have a probability of $\eulerE^{-1}$ to be at least~$1$, which dominates $\mathrm{Geom}(2q \cdot \eulerE)-1$.
    Since the success probability of this distribution is only different from~$q$ by constant factors, the result follows.
\end{proof}

Next we show \Cref{lem:infect_probability}, which lower bounds the probability of an active star $S_1$ activating an adjacent inactive star $S_2$ within a time interval of constant length. For the proof we split up the event into multiple events and lower bound each of their probabilities.

\infectProbability*

\begin{proof}
    We split~$E$ into five events that when they all occur imply $E$. In detail, we argue why it is likely that the center of $S_2$ becomes susceptible, then the center of $S_1$ becomes susceptible and then infected, infects the center of $S_2$, and then this infection activates~$S_2$. In order to assure that all of these events occur, we define each event to include the previous event, and then we bound its probability conditional on the previous one.

    Note that some of the events may occur immediately or in a different way as discussed above. For example, the center of $S_2$ may already start susceptible or it may be infected by some other vertices than the center of $S_1$. We ignore these additional possibilities as they just increase the probabilities of the events we aim to lower-bound.

    The first event, $E_1$, is the event that there exists a time $t_1 \in [0,1/\varrho]$ such that the center of $S_2$ is not immune at $t_1$. As the center of $S_2$ loses its immunity with a clock of rate~$\varrho$, we get $\Pr{E_1} \geq 1 - \eulerE^{-1}$.

    The next events assure that the center of $S_1$ becomes infected after $E_1$ occurred. Event~$E_2$ describes that $E_1$ occurs and that there exists a time $t_2 \in [t_1, t_1+1/\varrho]$ such that the center of $S_1$ is not immune. As it loses its immunity with a clock of rate $\varrho$, we get $\Pr{E_2}[E_1] \geq 1 - \eulerE^{-1}$.

    Let $E_3$ be the event that $E_2$ occurs and that there exists a time $t_3 \in [t_2, t_2+1/(1+\lambda)]$ such that the center of $S_1$ is infected at $t_3$. As $\lambda \in \bigOmega{\ell^{-1/2+\varepsilon}}$, by \Cref{lem:active_geom} the star stays active for a time polynomial in $\ell$ with probability at least $3/4$. Recovering all leaves takes w.h.p. at most a logarithmic time in $\ell$, therefore with a probability of at least $1/2$, the center of $S_1$ becomes infected again before $S_1$ becomes inactive. Conditioned on $S_1$ becoming infected again, the time it takes until $S_1$ becomes infected is dominated by an exponential distribution with rate $1+\lambda$. The reason for this is that for each leaf, the first trigger of it recovering and it infecting the center occurs at rate $1+\lambda$ independent of the other leaves and which of its two clocks triggers first. With the condition that the center of $S_1$ becomes infected again, there exists a leaf for which its infection clock triggers first. The time until this occurs follows an exponential distribution with rate $1+\lambda$. Other leaves may infect the center of~$S_1$ before that, but the time until the center is infected is dominated by this exponential distribution with rate $1+\lambda$. Hence $\Pr{E_3}[E_2] \geq (1 - \eulerE[-1])/2$.

    Event~$E_4$ denotes that $E_3$ occurs and there is a time $t_4 \in [t_3,t_3 + 1/(1+\lambda)]$ such that the center of $S_2$ is infected at $t_4$. As the center of $S_1$ starts infected and the center of~$S_2$ is susceptible, an infection occurs if the infection clock triggers before the recovering clock of the center of $S_1$. As both of those clocks are exponentially distributed, the first of them triggers after a time exponentially distributed with rate $1+\lambda$, and the probability of it being the infection trigger is $\lambda/(1+\lambda)$. Hence $\Pr{E_4}[E_3] \geq (1 - \eulerE^{-1})\lambda/(1+\lambda)$.

    Last, $E_5$ is the event that $E_4$ occurs and there exists a time $t_5 \in [t_4,t_4+1]$ such that $S_2$ is active at $t_5$. As the center of $S_2$ recovers at rate~$1$, there is a probability of $1 - \eulerE^{-1}$ that it recovers within the given time frame. By \Cref{pre:center_infected}, there is a probability of at least $1-\varepsilon$ that the star becomes active before the center recovers. Hence, $\Pr{E_5}[E_4] \geq (1 - \eulerE^{-1})(1-\varepsilon)$. Note that \Cref{pre:center_infected} assumes that, for a constant $c = \frac{\varrho}{4(2 + \varrho)}$, there are always $2c\ell$ leaves that are not recovered. By \cite[Lemma~4.4]{gobel2025gradually} that is true asymptotically almost surely, hence this additional condition only changes the probability by a factor $1-\smallO{1}$, and we obtain $\Pr{E_5}[E_4] \geq (1 - \eulerE^{-1})(1-\varepsilon)$ without the condition that there are always $2c\ell$ leaves that are not recovered by choosing a slightly bigger~$\varepsilon$.

    We conclude by bounding the probability of $E$ based on the probabilities of the events $E_1$ to $E_5$. Note that these events all describe something occurring in a time frame that starts after the previous event occurs, and note that appending all of these time frames yields a time frame of length at most $2(1 + 1/\varrho + 1/(1+\lambda))$. Therefore, event $E_5$ implies $E$, which means we have
    \begin{align*}
        \Pr{E} & \geq \Pr{E_5}                                                                                                                                                          \\
               & = \Pr{E_1} \cdot \Pr{E_2}[E_1] \cdot \Pr{E_3}[E_2] \cdot \Pr{E_4}[E_3] \cdot \Pr{E_5}[E_4]                                                                             \\
               & \geq (1 - \eulerE^{-1}) \cdot (1 - \eulerE^{-1}) \cdot 2^{-1}(1 - \eulerE[-1]) \cdot (1 - \eulerE^{-1})\lambda(1+\lambda)^{-1} \cdot (1 - \eulerE^{-1})(1-\varepsilon) \\
               & \in \bigOmega{\lambda}.\qedhere
    \end{align*}
\end{proof}

Using \Cref{lem:active_geom,lem:infect_probability}, we show \Cref{thm:survival}.

\Survival*

\begin{proof}
    We consider the discrete version of the simplified meta process $X_{\tau(i)}$ (see \Cref{def:meta-process}) which counts the number of active stars over time and discretizes it to times where this value changes. We show that this process dominates a biased gambler's ruin problem. By bounding the time this gambler's ruin instance needs to hit its lower end (\Cref{pre:gamblersRuin}) and bounding the time each step takes, we obtain a bound on the survival time.

    The process $X_{\tau(i)}$ changes its values by~$1$ up and down in the interval $[0 .. k]$ until~$T$. Note that at time~$T$, we have $X_T = 0$, hence,~$T$ is not reached before $X_{\tau(i)}$ drops to~$0$. By \Cref{pre:center_infected}, there is a probability of at least $1-\varepsilon$ that starting with $X_0=0$, the random variable $X_{\tau(1)}$ reaches~$1$ and~$T$ is not reached before.

    Now consider a state at step $i\in \N$ with $X_{\tau(i)} \in [1 .. k-1]$. We aim to bound the probability of $X_{\tau(i+1)}$ increasing and decreasing by~$1$. To this end, we consider a time interval $I \subset \R_{\geq 0}$ of length $\delta = \lceil 2(1 + 1/\varrho + 1/(1+\lambda))\rceil$, and we bound the probability of $\tau(i+1)\in I$ and $X_{\tau(i+1)}$ increasing or decreasing in that interval. If $\tau(i+1)$ is not in $I$, we use the same argument on the next interval of the same length iteratively.

    As we bound in \Cref{lem:active_geom} the time that a star stays active by a geometric random variable that is independent of the time of other stars, we bound the probability of~$X_{\tau(i+1)}$ increasing or decreasing by~$1$, independent of what happened before. By \Cref{lem:active_geom}, there exists a $p' \in \bigTheta{(\infectionRate^2\ell)^{-(1-\varepsilon)\deimmunizationRate}}$ such that the time until an active star becomes inactive dominates a random variable $A \sim \mathrm{Geom}(p')-1$. At time~$\tau(i)$, there are fewer than~$k$ active stars. Hence, we upper-bound the probability that one of these stars becomes inactive during the interval~$I$ by $k(\delta+1) \cdot p'$.

    In order to lower-bound the probability that a new star becomes active, we first assume that no star becomes inactive during the time interval~$I$, which happens with probability at least $1-k(\delta+1) \cdot p'$. Note that because the star centers induce a connected subgraph, $X_{\tau(i)} \in [1 .. k-1]$ implies that there is a center of an active star connected to a center of an inactive star. By \Cref{lem:infect_probability}, there is a constant $c \in \R_{>0}$ such that this inactive star becomes active with probability at least $c \lambda$. Hence,~$X_{\tau(i)}$ increases by~$1$ during the time interval~$I$ with probability at least $(1-k(\delta+1) \cdot p') c \lambda$. Using the same argument iteratively on the next time interval of the same length, we get that the probability of $X_{\tau(i)}$ increasing by~$1$ in the next step divided by the probability of it decreasing is at least $\frac{(1-k(\delta+1) \cdot p') c \lambda}{k(\delta+1) \cdot p'}$. Hence, the process $(X_{\tau(i)})_{i \in \N}$ dominates a biased gambler's ruin instance $Y$ with the same ratio between going up and down by~$1$. By noting that $\delta$ is upper- and lower-bounded by constants and that $\lambda \leq 1$, we get that $\lambda \in \smallOmega{kp'}$ implies that $1-k(\delta+1) \cdot p'$ converges to~$1$,
    Hence, we get that there is a $p \in \bigTheta{(\infectionRate^2\ell)^{-(1-\varepsilon)\deimmunizationRate}}$ such that~$Y$ has a ratio between going up and down by~$1$ of $\frac{\lambda}{kp}$ when $\lambda \in \smallOmega{kp}$.

    Using the second inequality of \Cref{pre:gamblersRuin}, we get that the probability of~$Y$ reaching~$k$ starting from~$1$ is $\frac{1-kp/\lambda}{1-(kp/\lambda)^k}$. For $\lambda \in \smallOmega{kp}$, this term converges to~$1$. Hence, for sufficiently large~$k$, there is a constant probability of at least~$1/2$ of~$Y$ reaching~$k$.

    Using the first inequality of \Cref{pre:gamblersRuin}, the probability of dropping down from $k-1$ to $0$ before reaching $k$ is $\frac{1-\lambda/kp}{1-(\lambda/kp)^k}$, which is in $\bigTheta{\left(kp/\lambda\right)^{k-1}}[\big]$, as $\lambda \in \smallOmega{kp}$ implies that the $1$s are lower-order terms. As~$X_{\tau(i)}$ dominates~$Y$, its probability of dropping down from $k-1$ to~$0$ before reaching~$k$ is even lower. Therefore, the expected number of steps until $X_{\tau(i)}$ reaches~$0$ is in $\bigOmega{\left(\lambda/kp\right)^{k-1}}[\big]$. As stars become inactive at rate at most $p$, dropping down from $k$ to $k-1$ takes in expectation at least $1/(kp)$ time. Therefore, the process $(X_t)_{t \in \R_{\geq 0}}$ takes in expectation at least $\bigOmega{(kp)^{-1}\left(\lambda/kp\right)^{k-1}}[\big]$ time.
\end{proof}

\section{Super-Polynomial SIRS Survival Time in Scale-Free Graph Models}\label{ap:real_world}

\compare*

\begin{proof}
    The expander bound yields a super-polynomial expected survival time once $\infectionRate \geq \infectionConstant \numberOfVertices^{(\hyperbolicExponent-3)/2}$. In contrast, \Cref{cor:scale_free_survival} yields such an expected time once $\lambda \in \bigOmega{n^{-\frac{1}{(\gamma-1)(2+1/\varrho)}+\varepsilon}}$. As the $\varepsilon$ is an arbitrarily small constant, the connected-star bound is smaller if and only if the exponent without the $\varepsilon$ is strictly smaller than the exponent from the expander bound. We observe that
    \begin{align*}
                        &  & -\frac{1}{(\gamma-1)(2+1/\varrho)}         & < (\gamma-3)/2                        \\
        \Leftrightarrow &  & (\gamma -3)(\gamma-1)                      & > -\frac{2}{2+1/\varrho}              \\
        \Leftrightarrow &  & \gamma^2 -4\gamma +\frac{2}{2+1/\varrho}+3 & >0                                    \\
        \Leftrightarrow &  & \gamma                                     & > 2 + \sqrt{1- \frac{2}{2+1/\varrho}} \\
        \Leftrightarrow &  & \gamma                                     & > 2 + \sqrt{\frac{1}{2\varrho+1}}.
    \end{align*}

    Note that the third-to-last inequality is quadratic and therefore yields two regimes for~$\gamma$ in which it is fulfilled. However, the second regime is strictly below~$2$, and we only consider $\gamma \in (2, 3)$. Hence we get that our threshold based on connected stars is smaller than the threshold based on expansion if and only if $\gamma > 2 + (2\varrho+1)^{-1/2}$.
\end{proof}